\documentclass[12pt]{article}
\usepackage{amsmath,amssymb}
\usepackage{amsthm}

\def\IC{\mathcal{I}}

\def\MC{\mathcal{M}}

\def\FC{\mathcal{F}}

\def\E{\mathbf{E}}
\def\N{\mathbf{N}}
\def\P{\mathbf{P}}
\def\R{\mathbf{R}}

\def\1{\mathbf{1}}

\def\al{\alpha}
\def\be{\beta}
\def\pa{\partial}

\def\de{\delta}
\def\ga{\gamma}

\newtheorem{prop}{Proposition}[section]
\newtheorem{theorem}{Theorem}[section]

\newtheorem{remark}{Remark}

\newcommand{\la}{\lambda}
\newcommand{\si}{\sigma}

\newcommand{\om}{\omega}
\newcommand{\Ga}{\Gamma}

\begin{document}
\title{Stochastic monotonicity and duality of $k$th order with application to put-call symmetry of powered options
\thanks{To appear: Journal of Applied Probability 52:1 (March 2015)}
}
\author{Vassili N. Kolokoltsov\thanks{Department of Statistics, University of Warwick,
 Coventry CV4 7AL UK, associate member of ZIF Bielefeld and IPI RAN Moscow,
  Email: v.kolokoltsov@warwick.ac.uk}}
\maketitle

\begin{abstract}
We introduce a notion of $k$th order stochastic monotonicity and duality that allows one to unify the notion used in insurance mathematics (sometimes refereed to as Siegmund's duality) for the study of ruin probability and the duality responsible for the so-called put - call symmetries in option pricing. Our general $k$th
order duality can be financially interpreted as put - call symmetry for powered options. The main objective of the present paper
is to develop an effective analytic approach to the analysis of duality leading to the full characterization of $k$th order duality of Markov processes in terms of their generators, which is new even for the well-studied case of put -call symmetries.
\end{abstract}

{\bf Key words:} stochastic monotonicity, stochastic duality, generators of dual processes, dual semigroup,
 put - call symmetry and reversal, powered and digital options, straddle.

{\bf Mathematics Subject Classification:} 60J25

\section{Introduction}

\subsection{Main objectives}

A real-valued Markov process $X_t^x$ is called stochastically monotone if $\P(X^x_t\ge y)$ is a non-decreasing function of $x$ for any $y$.
Siegmund's theorem (see \cite{Sieg}) states that if $X_t^x$ is stochastically monotone and $\P(X^x_t\ge y)$ is a right continuous function of $x$ for any $y$, then there exists a Markov process $Y_t^y$, called dual to $X_t^x$ such that
\begin{equation}
 \label{eqdefstanddual2ndmonrep}
\P (Y_t^y \le x)=\P (X_t^x \ge y)
\end{equation}
holds. This condition can be also rewritten as
\begin{equation}
 \label{eqdefstanddual2ndmonrep1}
\E \, \theta (x-Y_t^y)=\E \, \theta (X_t^x - y),
\end{equation}
where $\theta$ is the step function
\begin{equation}
 \label{eqdefstepfun}
\theta (x)=\left\{
\begin{aligned}
& 1, \quad x\ge 0 \\
& 0, \quad x<0.
\end{aligned}
\right.
\end{equation}

In the theory of option pricing, Markov processes $X_t^x,Y_t^y$ are said to satisfy the put-call symmetry relation if
\begin{equation}
\label{eqdefputcallsymm}
\E  (x-Y_t^y)_+=\E (X_t^x - y)_+
\end{equation}
holds. Looking at \eqref{eqdefputcallsymm} and \eqref{eqdefstanddual2ndmonrep1} suggests to introduce a general notion that includes these two dualities
as particular cases. Namely, let us say that
a Markov process $Y_t^y$ is dual to $X_t^x$ of order $k$, $k\in \R$, if
\begin{equation}
\label{eqdefputcallsymmgen}
\E  (x-Y_t^y)_+^{k-1}=\E (X_t^x - y)_+^{k-1}.
\end{equation}
The cases $k=2$ and $k=1$ correspond to \eqref{eqdefputcallsymm} and \eqref{eqdefstanddual2ndmonrep1} respectively
(with a natural convention that $x_+^0=\theta (x)$).
These dualities have also a clear financial interpretation describing symmetries between powered European options,
the case $k=1$ standing for a symmetry between digital options.

The aim of this paper is to fully characterize Markov processes satisfying  \eqref{eqdefputcallsymmgen} in terms of their generators,
paying special attention to processes that are martingales, as such processes appear in risk-neutral evaluation. This
characterization seems to be new even for the standard put-call symmetry \eqref{eqdefputcallsymm}, though the important particular cases
of underlying price processes being L\'evy processes or processes with price independent compensator are well studied, see \cite{AndCarr}, \cite{FaMor}.
We shall also extend the theory to time non-homogeneous processes, related notion of duality being referred to in \cite{AndCarr} as the put - call reversal.

We shall not pay attention to positivity of our martingales (which should be of course the case for realistic price processes), as this problem can be handled separately from the discussion of duality, either  by insuring that the origin is not attainable, or by directly working with exponents.

We shall also not address the issues arising at boundary points, as this development is treated separately in \cite{KoRui},
in connection with the problems from
insurance mathematics, where this question becomes crucial (ruin problem, see \cite{As98}, \cite{AsPi}, \cite{AsPe}, \cite{SigRy}, \cite{Dje93}),
because precisely
the absorption rates for attainable origin becomes there the most important quantity to study.

\subsection{Plan of the paper}

In Section \ref{secanaldual}
we introduce our analytic approach to the analysis of duality of Markov processes via their generators.
In Section  \ref{secxeampleopt} we present some simplest examples of duality arising from our results.
In Section \ref{secmonotkthorder} we extend the notion of stochastic monotonicity to arbitrary orders and prove the corresponding extension of
Siegmund's theorem linking stochastic monotonicity and duality.
In Section \ref{secmainres} we obtain our main results on the characterization of duality of one-dimensional Markov
processes via their generators. The last section is devoted to the extension of the theory to time-nonhomogeneous Markov processes.
In appendix we summarize in appropriate form some crucial facts about fractional derivatives used in the main text.

\subsection{Bibliographical comments}

Duality of Markov processes is an important topic in probability, see e.g. \cite{Lig} for an extensive introduction to the subject,
with special stress on interacting particles, see also \cite{Chenbook04} for the related study of stochastic monotonicity.
Paper \cite{Sieg} initiated systematic research of the duality based on stochastic order.
For crucial applications of duality in super-processes we can refer to \cite{EK} and \cite{Myt}.
The duality for general recursions and the duality for discrete Markov chains are developed in \cite{AsSig} and
\cite{HuiMar} respectively. For stochastic monotonicity and duality of birth and death processes we refer to \cite{VanDo}.

For general introduction to intertwining and many examples related to L\'evy processes see
\cite{Biane}, \cite{CaPeYo}, \cite{HiYo}, \cite{PPTS}, \cite{Dube} and references therein.

The subject of put - call symmetry was initiated in  \cite{Bates88}, \cite{Bart} and attracted since then lots of attention.
We can refer to papers \cite{CaLe}, \cite{MolSch}, \cite{EbPaSh} for detailed reviews of recent developments.
Let us mention specifically papers \cite{AndCarr}, \cite{FaMor}, where put - call symmetry was analyzed for markets based on diffusions with price independent jumps
and L\'evy processes respectively. Paper \cite{CaCh96} developed the theory for American options and papers
\cite{HenHoSh}, \cite{HenWo} for Asian options.  Paper \cite{EbPaSh} characterizes the symmetry in terms of semimartingale characteristics of general
dual semimartingales related by the dual martingale measures.
An important recent development concerns the study of quasi self-dual process,
 which relates the conditional symmetry properties of both their ordinary as well as their stochastic logarithms,
 see \cite{RheiSchm1}, \cite{RheiSchm2}.

For the application to insurance mathematics we refer to \cite{As98}, \cite{AsPi}, \cite{SigRy}, \cite{Dje93}) and references therein.

The approach to the study of Siegmund's duality via generators was initiated in \cite{Ko03} and continued in \cite{Ko10}, see
also monograph  \cite{Kobook11}.

\section{Analytic approach to the analysis of duality}
\label{secanaldual}

\subsection{Definition of stochastic duality}

Let us first recall the standard definition of duality of Markov processes.

Let $X_t^x$ and $Y_t^y$ be two  Markov processes
(small $x,y$ here and in what follows stand for the initial points)
with values in possibly different Borel spaces $X$ and $Y$. Then $Y$ is called dual to $X$ with respect to a Borel function $f$ on $X\times Y$,
or shortly $f$-dual, if
 \begin{equation}
 \label{eqdefstanddual2nd}
\E f(x,Y_t^y)=\E f(X_t^x,y)
\end{equation}
for all $x\in X, y\in Y$, where $\E$ on the left hand side and the right hand side correspond to the distributions
of the processes $Y_t^y$ and $X_t^x$ respectively.

An important example is given by the duality equation
\begin{equation}
 \label{eqdefstanddual2ndmon}
\P (Y_t^y \le x)=\P (X_t^x \ge y),
\end{equation}
where $\ge$ is a partial order. This is a particular case of  \eqref{eqdefstanddual2nd}
 with $f(x,y)=\1_{\{x\ge y\}}$ (we denote here and in what follows by $\1_M$ the indicator function of the set $M$).

From the point of view of the general definition of $f$-duality, duality of $k$th order given by \eqref{eqdefputcallsymmgen} corresponds to $f_k$-duality
for $f_k(x,y)=(x-y)_+^{k-1}$, where $x_{\pm}=\max(0,\pm x)$.

\subsection{Analytic counterpart of duality}

For a metric space $X$ we denote by $B(X), C(X), \MC(X)$ the Banach spaces of bounded measurable functions, bounded continuous functions
and bounded signed Borel measures, first two spaces equipped with the sup-norm and the last one with the total variation norm.
If $X$ is locally compact, $C_{\infty}(X)$ denotes the closed
subspace of  $C(X)$ of functions vanishing at infinity.
The standard duality between $B(X)$ and $\MC(X)$ is given by the integration: $(f, \mu)=\int_X f(x) \mu (dx)$.

 By a signed (stochastic) kernel from $X$ to $Y$ we mean, as usual, a function of two variables $p(x, A)$, where $x\in X$ and $A$ are Borel subsets of $Y$ such that $p(x,.)$ is a bounded signed measure on $Y$ for any $x$ and $p(., A)$ is a Borel function for any Borel set $A$. We say that this kernel is bounded if $\sup_x \|p(x,.)\|<\infty$.

Any bounded kernel specifies an integral operator $B(Y) \to B(X)$ via the formula
\[
Ug(x)=\int_Y g(z) p(x, dz).
\]
The standard dual operator $U'$ is defined as the operator $\MC(X)\to \MC(Y)$ specified by the duality relation
\[
(f,U'\mu)=(Uf, \mu),
\]
or explicitly as
\[
U'\mu(dy)=\int_X p(x, dy) \mu (dx).
\]

A bounded linear operator $U^{D(f)}$ in $B(Y)$ (or $C(Y)$  or $C_{\infty}(Y)$) is said to be $f$-dual to
a bounded linear operator $U$ in $B(X)$ (or $C(X)$  or $C_{\infty}(X)$) if, for any $x,y$,
  \begin{equation}
 \label{eqdeffTdualanal}
(U^{D(f)} f(x,.))(y) =(U f(.,y))(x).
\end{equation}

Let us say that a function $f$ on $X\times Y$
 {\it separates measures on} $X$ if,
for any $Q_1,Q_2 \in \MC(X)$, there exists $y \in Y$ such that $\int f(x,y) Q_1(dx)\neq \int f(x,y) Q_2(dx)$.
  If this is the case, the integral operator $F=F_f:\MC(X)\to B(Y)$ given by
\begin{equation}
\label{eqdefdualsemigroupsF}
(FQ)(y)=\int f(x,y) Q(dx)
\end{equation}
is an injective bounded operator, so that the linear inverse $F^{-1}$ is defined on the image $F(\MC(X))$.
Let us say that the function $FQ$ is $f$-generated by $Q$.

Our analysis will be based on the following simple but crucial observation.

\begin{prop}
\label{propdualop}
Let $f$ be a bounded measurable function separating measures on $X$ and $U$ an integral operator in $B(X)$ with a bounded signed kernel $p$.

(i) Suppose $U^{D(f)}$ is an integral operator with a bounded kernel $p^{D(f)}(y,dz)$ satisfying \eqref{eqdeffTdualanal}. Then
the action of $U^{D(f)}$ on $F(\MC(X))$ is given by the equation
\begin{equation}
\label{eqpropdualop}
U^{D(f)}=F \circ U' \circ F^{-1},
\end{equation}
or, equivalently, $U^{D(f)}$ satisfies the intertwining relation
\begin{equation}
\label{eqpropdualop1}
U^{D(f)} \circ F = F \circ U'.
\end{equation}

(ii) Let us define an operator $U^{D(f)}$ on $F(\MC(X))$ by the the linear extension of relation \eqref{eqdeffTdualanal}, that is, by the equation
\begin{equation}
\label{eqdeffTdualanalonF}
(U^{D(f)} F(Q))(y)= (U^{D(f)} \int_X f(x,.) Q(dx))(y) =\int_X (Uf(.,y))(x) Q(dx).
\end{equation}
Then $U^{D(f)}$
is well defined on $F(\MC(X))$ and \eqref{eqpropdualop} holds.
\end{prop}

\begin{proof}
(i) Let $g\in F(\MC(X))$ be given by $g(y)=\int f(x,y) Q_g(dx)$. Then
\[
 U^{D(f)}g(y)=\int_Y g(z) p^{D(f)}(y, dz),
 \]
 which by Fubini's theorem rewrites as
 \[
 U^{D(f)}g(y)=\int_X \left(\int_Y f(x,z) p^{D(f)}(y, dz)\right)  Q_g(dx),
 \]
 and consequently as \eqref{eqdeffTdualanalonF} with $Q=Q_g$. Hence
 \[
 U^{D(f)}g(y)=\int_X\int_Y f(z,y) p(x,dz)Q_g(dx)=\int_Y f(z,y) \tilde Q (dz),
\]
with
\[
\tilde Q(dz)=\int p(x,dz) Q_g(dx).
\]
Thus $U^{D(f)}g$ is $f$-generated by $\tilde Q=U'Q_g$, as required.

(ii) Instead of using Fubini's theorem, we start with \eqref{eqdeffTdualanalonF} by definition. The remaining calculations are the same.
\end{proof}

\begin{remark}
For discrete Markov chains, Proposition \ref{propdualop} was proved in  \cite{HuiMar}.
\end{remark}

\subsection{Application to semigroups and Markov processes}

Representation \eqref{eqpropdualop} has the following direct implication for the theory of semigroups.

\begin{prop}
\label{propdualsemi}
(i) Let $f$ be a bounded measurable function separating measures on $X$ and $T_t$ a semigroup of integral operators in $B(X)$ (or $C(X)$, or $C_{\infty}(X)$)
specified by the family of bounded signed kernel $p_t(x,dz)$ from $X$ to $X$.
Then the dual operators $T_t^{D(f)}$ (defined by \eqref{eqdeffTdualanalonF} with $U=T_t$) in $F(\MC(X))$ also form a semigroup and
\begin{equation}
\label{eqpropdualopesemig}
T_t^{D(f)}=F \circ T'_t \circ F^{-1}.
\end{equation}

(ii)
If the semigroup $T_t$ is generated by an operator $L$ in $C(X)$ defined on some invariant (under all $T_t$) domain $D$, then
\[
\left.\frac{d}{dt}\right|_{t=0}T_t^{D(f)}g= F \circ \left.\frac{d}{dt}\right|_{t=0} T' \circ F^{-1}g=F \circ L' \circ F^{-1}g,
\]
that is, the generator of the semigroup $T_t^{D(f)}$ is
\begin{equation}
\label{eqdualgengen}
L^{D(f)}=F \circ L' \circ F^{-1},
\end{equation}
with domain containing the image (under $F$) of the domain of $L'$.
\end{prop}

\begin{proof}
(i) This is straightforward from \eqref{eqpropdualop} and the standard obvious fact that $T_t'$ form a semigroup in $\MC(X)$.
(ii) Follows from (i).
 \end{proof}

\subsection{Duality for $f$ depending on the difference of its arguments}

The theory simplifies essentially if $f$ is translation-invariant, that is, $f$ depends only on the
difference of its arguments, $f(x,y)=f(y-x)$,
with some other function $f$ that we still denote by $f$ (with some ambiguity). In this case,
 the operator $F$ from \eqref{eqdefdualsemigroupsF}, applied  to a measure $Q$ with density $q$, takes the form
\begin{equation}
\label{eqdefdualsemigroupsFsep}
g(y)=(FQ)(y)=\int_{\R^d} f(y-x) q(dx),
\end{equation}
i.e. it becomes a convolution operator.
It is then well known that under appropriate regularity assumptions, $f$ is the fundamental solution of the pseudo-differential operator
$L_f$ with the symbol
\begin{equation}
\label{eqfundsolFourier1}
L_f(p)=\frac{1}{\hat f(p)},
\end{equation}
where
\[
\hat f(p)=\int e^{-ixp} f(x) dx
\]
is the Fourier transform of $f$.

Hence $g(y)$ from \eqref{eqdefdualsemigroupsFsep} solves the equation $L_fg=q$, so that $F^{-1}=L_f$.
Multidimensional examples with differential operators $L_f$ are given in \cite{KoRui}.

\section{Simplest examples}
\label{secxeampleopt}

Let $X_t^x$ be the stable-like Markov process with the Feller semigroup generated by the operator
\begin{equation}
\label{eqgenstablelikekonedim}
Lg(x)=\mp a(x)\frac{d^k}{dx^k}, \quad k \in (0,2],
\end{equation}
with a nonnegative continuously differentiable function $a(x)$ (see Appendix for the definition of fractional derivatives and integral
  operators used here and in what follows), where the signs $\mp$ correspond to the cases $k\in (0,1)$ and $k\in (1,2]$ respectively
  (for the trivial case $k=1$ the sign is non-essential).
Then, by \eqref{eqdualgengen}, the dual generator of order $k$ is given by
\[
L^{D_k}=\mp I_k^+ \frac{d^k}{d(-x)^k} a(x) \frac{d^k}{d(-x)^k}=\mp a(x)\frac{d^k}{d(-x)^k}
\]
(where we used definition \eqref{eqdeffracint1} and the properties of $I_k$ discussed before \eqref{eqintpartsfrac0}),
so that the dual process to $X_t^x$ of order $k$ is the process $Y_t^y$ generated by
\begin{equation}
\label{eqgenstablelikekonedimdual}
L^{D_k}g(x)=\mp a(x)\frac{d^k}{d(-x)^k},
\end{equation}
leading to the following result.

\begin{prop}
\label{propstablelikekonedimdual}
Let $X_t^x$, $Y_t^y$ be Markov processes generated by \eqref{eqgenstablelikekonedim} and \eqref{eqgenstablelikekonedimdual} respectively,
 with $k\in (0,2]$. Then
\begin{equation}
\label{eqpropdualbyLapfr1}
 \E (x-Y_t^y)_+^{k-1}=\E  (X_t^x-y)_+^{k-1}.
 \end{equation}
 \end{prop}

In financial terms this means that the price of the European powered call option
for the initial stock price $x$ and the strike $y$ equals the price of the European powered put option
for the initial stock price $y$ and the strike $x$ (discounting is supposed to be already included in the definition of processes $X_t^x$ and $Y_t^y$).

The most important cases are with $k\in (1,2)$, since then the corresponding Markov processes $X_t^x,Y_t^y$ are martingales,
and thus the expectation corresponds to a risk-neutral evaluation.

The case of the diffusions, that is $k=2$, is well known, see e.g. \cite{AndCarr}.

Similarly,
if $X_t^x$ is generated by
 \begin{equation}
\label{eqgenstablelikeksymonedim}
 Lg(x)=-a(x)\left|\frac{d}{dx}\right|^k,
 \end{equation}
 with $k \in (0,2]$, then the dual with respect to the function $f(x,y)=|x-y|^k$
has the generator
\[
 L^{D}=- \left|\frac{d}{dx}\right|^{-k} \left|\frac{d}{dx}\right|^k a(x) \left|\frac{d}{dx}\right|^k
 =-a(x)\left|\frac{d}{dx}\right|^k,
\]
which coincides with $L$. Consequently, $X_t^x$ is self-dual in this sense, leading to the following.

\begin{prop}
\label{propdualstraddle}
Let $X_t^x$ be the Markov process generated by
 \eqref{eqgenstablelikeksymonedim}
 with $k \in (0,2]$.
 Then
 \begin{equation}
\label{eqpropdualbyLapfr2}
 \E |y-X_t^x|^{k-1}=\E  |X_t^y-x|^{k-1}.
 \end{equation}
 \end{prop}

In the financial interpretation this means the self-symmetry of powered straddle spreads.

Similarly one can analyze symmetries linking various option spreads, though the conditions for underlying Markov processes
can become rather restrictive. For instance, let us consider a symmetry related to the so-called bull put spread,
 whose premium has the form (up to a linear equivalence)
 \begin{equation}
 \label{eqebullputspreadprem}
 f_{\al, \be} (x,y)= (x-y+\al)_+-(x-y+\be)_+
 \end{equation}
 (powered version can be analyzed analogously).
 The corresponding operator $F$ from \eqref{eqdefdualsemigroupsF} can be taken as
 $F=(T_{\al}-T_{\be})I^+_2$, where $T_c f(x)=f(x+c)$ denotes the shift, so that
 \[
 F^{-1}=\sum _{m=0}^{\infty} T_{\be -\al}^m T_{\al}^{-1}\frac{d^2}{dx^2}.
 \]
 Hence, for $L=a(x)d^2/dx^2$, we get
 \[
 L^Dg(x)=a(x+\al)\sum _{m=0}^{\infty} T_{\be -\al}^m - a(x+\be)\sum _{m=1}^{\infty} T_{\be -\al}^m,
 \]
 which equals $L$ if $a(x)$ is a $(\be -\al)$-periodic function. In this case we get the duality relation
 \begin{equation}
\label{eqpropdualbyLapfr1}
 \E  f_{\al, \be} (x, X_t^y)=\E f_{\al, \be} (X_t^x,y).
 \end{equation}

\section{Stochastic monotonicity and duality}
\label{secmonotkthorder}

Let us say that a Markov process $X_t^x$
with transition probabilities $p_t(x,dz)$ is stochastically
monotone of order $k>0$, if for any $t>0$, $y\in \R$, the derivative
\begin{equation}
\label{eq00thstochmonkthorder}
\frac{\pa ^{k}}{\pa x^{k}} \E (X_t^x-y)^{k-1}_+=\frac{\pa ^k}{\pa x^k} \int_{z\ge y} (z-y)^{k-1}_+ p_t(x,dz)
\end{equation}
exists in the sense of distribution and is a positive measure (this includes the assumption that $\E(X_t^x-y)_+^{k-1}$ is finite for all $y$).
 If $k\ge 1$, then an equivalent requirement is that,
for any $t>0$, $y\in \R$, the derivative
\begin{equation}
\label{eq0thstochmonkthorder}
\frac{\pa ^{k-1}}{\pa x^{k-1}} \E (X_t^x-y)^{k-1}_+=\frac{\pa ^{k-1}}{\pa x^{k-1}} \int_{z\ge y} (z-y)^{k-1}_+ p_t(x,dz)
\end{equation}
exists in the sense of distribution and is a non-decreasing function of $x$.

\begin{remark}
(i) If $k\ge 2$, this can be reformulated avoiding generalized functions  by saying that
the derivative
\[
\frac{\pa ^{k-2}}{\pa x^{k-2}} \int_{z\ge y} (z-y)_+^{k-1} p_t(x,dz)
\]
exists as an absolutely continuous function such that its first derivative (defined almost surely) is a non-decreasing function of $x$.
(ii) One can also formulate the notion of stochastic monotonicity of arbitrary order, avoiding generalized derivatives,
 in terms of the positivity of the increments of $k$th order of the function $\E (X_t^x-y)^{k-1}_+$ (as a function of $y$).
\end{remark}

Usual stochastic monotonicity corresponds to $k=1$.
The following result extends Siegmund's theorem to monotonicity of higher orders.

\begin{theorem}
\label{thstochmonkthorder}
A real-valued Markov process $X_t^x$ with transition probabilities $p_t(x,dz)$ and semigroup $T_t$
has a Markov dual of order $k \ge 1$ if and only if it is stochastically monotone of order $k$, satisfies the limiting relation
\begin{equation}
\label{eq1thstochmonkthorder}
\frac{\pa ^{k-1}}{\pa x^{k-1}} \int_{z\ge y} \frac{(z-y)^{k-1}_+}{\Ga (k)} p_t(x,dz) \to \left\{
\begin{aligned}
& 1, \quad x\to \infty, \\
& 0, \quad x\to -\infty,
\end{aligned}
\right.
\end{equation}
for all $y$
and, if $k=1$, the function $\int_{z\ge y} p_t(x,dz) =\P(X_t^x\ge y)$ is right continuous.
\end{theorem}

\begin{proof}
Let us first analyze Siegmund's case $k=1$ from our point of view.
Then the mapping $F: \MC(\R)\to B(\R)$ given by the corresponding equation \eqref{eqdefdualsemigroupsF} becomes the usual integration, that is
\[
FQ (y)=\int \theta (x-y) Q(dx)=\int_{x\ge y} Q(dx),
\]
whose image
consists of the left continuous functions (because we defined $\theta$ to be right continuous) of uniformly bounded variation (the total variation of $FQ$ being equal to the total variation norm
of $Q$) tending to zero at $+\infty$. By \eqref{eqpropdualopesemig}, for a $g= FQ$ with a finite measure $Q=-dg$, the corresponding dual semigroup becomes
\[
T_t^Dg(y) =F \circ T_t' \circ F^{-1} g(y)=-\int \left( \int_{z\ge y}  p_t(w, dz)\right) dg (w).
\]
Integrating by parts (notice that here it is crucial that the functions $\int_{z\ge y} p_t(x,dz)$ and $g$ are right continuous and
left continuous respectively, see \eqref{eqintpartsfrac1} with $k=1$), this rewrites as
\begin{equation}
\label{eq2thstochmonkthorder}
T_t^Dg(y) = \int g(w) d_w \int_{z\ge y}  p_t(w, dz),
\end{equation}
where $d_w \int_{z\ge y}  p_t(w, dz)=d_w \P(X_t^w\ge y)$ is the Stiltjes measure of the increasing function $\P(X_t^w\ge y)$.
Equation \eqref{eq2thstochmonkthorder} defines an integral operator with a positive stochastic kernel, which, taking into account
assumption \eqref{eq1thstochmonkthorder}, is in fact a probability kernel. Hence this operator extends naturally to a positivity preserving
conservative contraction in $B(\R)$ thus defining a Markov transition operator. Finally the Markov property (which is now equivalent to
 the Chapman-Kolmogorov equation or to the semigroup property of the operators $T_t^{D(f)}$) follows from Proposition \ref{propdualsemi}.

Now let $k>1$. Then the corresponding operator $F$, given by \eqref{eqdefdualsemigroupsF}, becomes (up to a constant multiplier) the
integration operator $I_k^+$ (see \eqref{eqdeffracint1}
from Appendix) and the corresponding function $f$ specifying duality is $(x-y)_+^{k-1}$.  Assuming $g$ belongs to the image $\IC_k^+$ of
$I_k^+$ (see  discussion after \eqref{eqdeffracint1}),
so that
\[
(I_k^+)^{-1} g(y)=\frac{d^k g(w)}{d(-w)^k}=Q(dw)
\]
is a measure from $\MC_k^+$,
one can integrate by parts (using \eqref{eqintpartsfrac1})
in the formula
\[
T_t^{D_k}g(y) =I_k^+ \circ T_t' \circ (I_k^+)^{-1} g(y)
\]
\[
= \int_{z\ge y} \frac{(z-y)_+^{k-1}}{\Ga (k)} \left(
 \int_{w\in \R} p_t(w, dz)\frac{d^k g(w)}{d(-w)^k} \right)
=\int \left( \int_{z\ge y} \frac{(z-y)_+^{k-1}}{\Ga (k)} p_t(w, dz)\right)
\frac{d^k g(w)}{d(-w)^k}
\]
(where the corresponding dual operators are marked by the subscript $D_k$)
leading to
\begin{equation}
\label{eq3thstochmonkthorder}
T_t^{D_k}g(y) = \int g(w) \frac{\pa ^k}{\pa w^k} \left( \int_{z\ge y} \frac{(z-y)_+^{k-1}}{\Ga (k)} p_t(w, dz)\right).
\end{equation}
This formula can be used to define a natural extension of \eqref{eq2thstochmonkthorder} (initially defined as a mapping $\IC_k^+ \to B(\R)$) as a positive integral operator. The proof is now completed as in case $k=1$.
\end{proof}

For $k\in (0,1)$ we assume a bit more regularity on the initial process $X_t^x$, which, on the one hand, is enough for most of the application and,
on the other hand, allows one to avoid rather subtle measure-theoretic problems. Consequently, the next result is obtained by literally the same proof
 as for $k>1$ above.

\begin{theorem}
\label{thstochmonkthorder}
Suppose a real-valued Markov process $X_t^x$ has bounded transition probability densities $p_t(x,z)$ for $t>0$. Then $X_t^x$
has a Markov dual of order $k \in (0,1)$ if and only if it is stochastically monotone of order $k$ and
measure \eqref{eq00thstochmonkthorder} has the total mass one.
\end{theorem}

\section{Characterization of duality in terms of generators}
\label{secmainres}

In this section we obtain our main results. Namely,
using the formula for $f$-dual generators \eqref{eqdualgengen} we explicitly calculate the dual for an arbitrary
Feller process. Let us first consider diffusions. Moreover, for clarity, we shall consider separately the cases of
integer and real $k$.

For an integer $k$, we shall denote by $C^k$ the space of $k$ times continuously differential functions on $\R$ (with bounded derivatives).
In what follows we are not aiming at the weakest possible assumption on the coefficients $a,b$, but assume as much regularity as
needed to get the most transparent formulas for dual operators. Also the convenient assumption of boundedness can be relaxed
by using the theory of diffusions with unbounded coefficients.

Let us consider a Feller diffusion $X_t^x$ with the Feller semigroup $T_t$ is generated by the operator
\begin{equation}
\label{eq1thstochmonkthorderdif}
Lg(x)=a(x) \frac{d^2}{dx^2}+b(x) \frac{d}{dx}
\end{equation}
with a nonnegative function $a(x)$.

The following fact is  a particular case of a more general multi-dimensional result from  \cite{KoRui}. If $a, b \in C^1$, then
the Markov dual process $Y_t^x$ of order $k=1$ exists and is a diffusion generated by the operator
\[
L^{D_k}g(x)=a(x) \frac{d^2}{dx^2}+[a'(x)-b(x)] \frac{d}{dx}.
\]

This allows us to exclude $k=1$ from the following arguments.

\begin{theorem}
\label{thstochmonkthorderdifint}
Suppose $k>1$ is an integer and $a, b \in C^k$.
The diffusion $X_t^x$ generated by \eqref{eq3thstochmonkthorderdif} is stochastically monotone of order $k$ if and only if the function
\[
\om_y(x)=\frac{\pa^{k-1}}{\pa x^{k-1}} \left[\frac{(x-y)^{k-2}}{\Ga (k-1)} b(x)+\frac{(x-y)^{k-3}}{\Ga (k-2)} a(x) \1_{k\neq 2}\right]
\]
is a non-decreasing function of $x\ge y$ for any $y$.
If this is the case and additionally
\begin{equation}
\label{eq2thstochmonkthorderdif}
\lim_{x\to +\infty} \om_y(x) =-(k-1)b'(y)-\frac12 (k-1)(k-2) a''(y)
\end{equation}
for all $y$, then the $k$th order Markov dual process $Y_t^y$ exists and is generated by the operator
\[
L^{D_k}g(y)=a(y) \frac{d^2}{dy^2}-[b(y)+(k-2)a'(y)] \frac{d}{dy}
\]
\begin{equation}
\label{eq3thstochmonkthorderdif}
+\int_y^{\infty} (g(x)-g(y)) \frac{\pa ^k}{\pa x^k}
\left[\frac{(x-y)^{k-2}}{\Ga (k-1)} b(x)+\frac{(x-y)^{k-3}}{\Ga (k-2)} a(x) \1_{k\neq 2}\right]\, dx.
\end{equation}
\end{theorem}

 \begin{remark}
If one has the inequality $\le $ rather than the equality in \eqref{eq2thstochmonkthorderdif},
then the dual $Y_t^y$ exists as a sub-Markov process.
\end{remark}

\begin{proof}
Let us do the calculations separately for the drift and diffusion parts of $L$ starting with the drift part.
By \eqref{eqdualgengen} and \eqref{eqintpartsfrac0} we have
\begin{equation}
\label{eq4thstochmonkthorderdif}
\left(b (y)\frac{d}{dy}\right)^{D_k}g(y)=-I_k^+ \left(\frac{d}{dy} \circ b(y)\right) \circ \frac{d^k}{d(-y)^k}g(y)
= \int_y^{\infty} \frac{(x-y)^{k-2}}{\Ga (k-1)} b(x) \frac{d^k}{d(-x)^k} g (x) \, dx.
\end{equation}
Integrating by parts first $k-2$ times (where boundary terms cancel) and then two times more yields
\[
\left(b (y)\frac{d}{dy}\right)^{D_k}g(y)
= \int_y^{\infty} \frac{d^{k-2}}{dx^{k-2}}\left[\frac{(x-y)^{k-2}}{\Ga (k-1)} b(x)\right] \frac{d^2 g(x)}{dx^2} \, dx
\]
\[
=-b(y) g'(y) +(k-1) b'(y)g(y)+\int_y^{\infty} g(x)\frac{\pa^k}{\pa x^k} \left[\frac{(x-y)^{k-2}}{\Ga (k-1)} b(x)\right] \, dx,
\]
which rewrites as
\[
\left(b (y)\frac{d}{dy}\right)^{D_k}g(y)=-b(y) g'(y) +\int_y^{\infty} (g(x)-g(y)) \frac{\pa^k}{\pa x^k} \left[\frac{(x-y)^{k-2}}{\Ga (k-1)} b(x)\right]\, dx
\]
\[
+g(y)\left[(k-1)b'(y)+\lim_{x\to \infty}\frac{\pa ^{k-1}}{\pa x^{k-1}} \left[\frac{(x-y)^{k-2}}{\Ga (k-1)} b(x)\right]\right],
\]
if the last limit is finite.

Similarly, if $k\neq 2$, we have
\begin{equation}
\label{eq5thstochmonkthorderdif}
\left(a (y)\frac{d^2}{dy^2}\right)^{D_k}g(y)=I_k^+ \left(\frac{d^2}{dy^2} \circ a(y)\right) \circ \frac{d^k}{d(-y)^k}g(y)
= \int_y^{\infty} \frac{(x-y)^{k-3}}{\Ga (k-2)} a(x) \frac{d^k}{d(-x)^k} g (x) \, dx,
\end{equation}
which by the integration by parts rewrites as
\[
-\int_y^{\infty} \frac{d^{k-3}}{dx^{k-3}}\left[\frac{(x-y)^{k-3}}{\Ga (k-2)} a(x)\right] \frac{d^3g(x)}{dx^3} \, dx
\]
\[
=a(y)g''(y)-(k-2) a'(y) g'(y) +\frac12 (k-1)(k-2) a''(y)g(y)+\int_y^{\infty} g(x)\frac{\pa^k}{\pa x^k} \left[\frac{(x-y)^{k-3}}{\Ga (k-2)} a(x)\right] \, dx,
\]
or finally as
\[
=a(y)g''(y)-(k-2) a'(y) g'(y) +\int_y^{\infty} (g(x)-g(y)) \frac{\pa^k}{\pa x^k} A_k(x,y)\, dx
\]
\[
+g(y) \left[ \frac12 (k-1)(k-2) a''(y) + \lim_{x\to \infty}\frac{\pa^{k-1}}{\pa x_{k-1}} \left[\frac{(x-y)^{k-3}}{\Ga (k-2)} a(x)\right]\right],
\]
if the limit on the r.h.s. exists. Summing up the formulas for the dual operators to the drift and diffusive parts
 leads to \eqref{eq3thstochmonkthorderdif} under the condition \eqref{eq2thstochmonkthorderdif}.
\end{proof}

The extension to non-integer $k$ goes as follows.

\begin{theorem}
\label{thstochmonkthorderdifre}
Suppose $k>0$ and $k\neq 1,2$, and $a \in C^{[k]+2}(\R)$, $b \in C^{[k]+1}(\R)$ (here $[k]$ denotes the integer part of a number $k$).
Let us define, for $x\ge y$,  the functions
\[
B_k(x,y)=\frac{1}{\Ga (k-1)} [b(x)-b(y)-b'(y)(x-y)](x-y)^{k-2},
\]
\[
A_k(x,y)=\frac{1}{\Ga (k-2)} [a(x)-a(y)-a'(y)(x-y)-\frac12 a''(y)(x-y)^2](x-y)^{k-3}.
\]
The diffusion $X_t^x$ generated by \eqref{eq3thstochmonkthorderdif}
is stochastically monotone of order $k$ if and only if the function
\[
\frac{\pa^{k-1}}{\pa x^{k-1}} (B_k(x,y)+A_k(x,y))
\]
is a non-decreasing function of $x\ge y$ for any $y$.
If this is the case and additionally
\begin{equation}
\label{eq2thstochmonkthorderdifre}
\lim_{x\to +\infty} \frac{\pa^{k-1}}{\pa x^{k-1}} (B_k(x,y)+A_k(x,y)) =-(k-1)b'(y)-\frac12 (k-1)(k-2) a''(y)
\end{equation}
for all $y$, then the $k$th order Markov dual process $Y_t^y$ exists and is generated by the operator
\[
L^{D_k}g(y)=a(y) \frac{d^2}{dy^2}-[b(y)+(k-2)a'(y)] \frac{d}{dy}
\]
\begin{equation}
\label{eq3thstochmonkthorderdifre}
+\int_y^{\infty} (g(x)-g(y)) \frac{\pa ^k}{\pa x^k} (B_k(x,y)+A_k(x,y))\, dx.
\end{equation}
\end{theorem}

\begin{proof}
If $k>1$, formula \eqref{eq4thstochmonkthorderdif} remains valid, but we rewrite it now as
\[
\int_y^{\infty} \left[ \frac{(x-y)^{k-2}}{\Ga (k-1)} (b(y)+b'(y)(x-y))+ B_k(x,y)\right] \frac{d^k}{d(-x)^k} g (x) \, dx.
\]
Using fractional integration-by-parts formulas \eqref{eqintpartsfrac1} and \eqref{eqintpartsfrac00} yields
\begin{equation}
\label{eq4thstochmonkthorderdifre}
\left(b (y)\frac{d}{dy}\right)^{D_k}g(y)=-b(y) g'(y) +(k-1) b'(y)g(y)+\int_y^{\infty} g(x)\frac{\pa^k}{\pa x^k} B_k(x,y)\, dx.
\end{equation}
Notice that the measure in the last integral does not have an atom at $y$ (in fact the function $B_k(x,y)$ was introduced specifically
in order to be able to single out such a measure, corresponding boundary terms being written explicitly). This can also be rewritten as
\[
\left(b (y)\frac{d}{dy}\right)^{D_k}g(y)=-b(y) g'(y) +\int_y^{\infty} (g(x)-g(y)) \frac{\pa^k}{\pa x^k} B_k(x,y)\, dx
\]
\begin{equation}
\label{eq5thstochmonkthorderdifre}
+g(y)\left[(k-1)b'(y)+\lim_{x\to \infty}\frac{d^{k-1}}{dx^{k-1}} B_k(x,y)\right],
\end{equation}
if the last limit is finite.

If $k\in (0,1)$, a small modification is required. Namely, in this case, instead of \eqref{eq4thstochmonkthorderdif}, we get
\[
\left(b (y)\frac{d}{dy}\right)^{D_k}g(y)
= -\int_y^{\infty} \frac{(x-y)^{k-1}}{\Ga (k)} \frac{\pa}{\pa x} \left(b(x) \frac{d^k}{d(-x)^k} g (x)\right) \, dx.
\]
Before integration by parts we have to add and subtract $b(y)$ from $b(x)$ leading to
\[
\left(b (y)\frac{d}{dy}\right)^{D_k}g(y)=-b(y) g'(y)
-\int_y^{\infty} \frac{(x-y)^{k-1}}{\Ga (k)} \frac{\pa}{\pa x} \left[(b(x)-b(y)) \frac{d^k}{d(-x)^k} g (x)\right] \, dx.
\]
Now we can integrate by parts yielding
\[
\left(b (y)\frac{d}{dy}\right)^{D_k}g(y)=-b(y) g'(y)
+\int_y^{\infty} \frac{(x-y)^{k-2}}{\Ga (k-1)} \left[(b(x)-b(y)) \frac{d^k}{d(-x)^k} g (x)\right] \, dx,
\]
which again turns to \eqref{eq4thstochmonkthorderdifre} and consequently to \eqref{eq5thstochmonkthorderdifre}.

Similarly \eqref{eq5thstochmonkthorderdif} remains valid for $k>2$, and we rewrite it now as
\[
\int_y^{\infty} \left[\frac{(x-y)^{k-3}}{\Ga (k-2)} (a(y)+a'(y)(x-y)+\frac12 a''(y)(x-y)^2) + A_k(x,y)\right] \frac{d^k}{d(-x)^k} g (x) \, dx,
\]
or using fractional integration-by-parts formulas \eqref{eqintpartsfrac1} and \eqref{eqintpartsfrac00} as
\[
a(y)g''(y)-(k-2) a'(y) g'(y) +\frac12 (k-1)(k-2) a''(y)g(y)+\int_y^{\infty} g(x)\frac{\pa^k}{\pa x^k} A_k(x,y) \, dx
\]
\[
=a(y)g''(y)-(k-2) a'(y) g'(y) +\int_y^{\infty} (g(x)-g(y)) \frac{\pa^k}{\pa x^k} A_k(x,y)\, dx
\]
\[
+g(y) \left[ \frac12 (k-1)(k-2) a''(y) + \lim_{x\to \infty}\frac{\pa^{k-1}}{\pa x_{k-1}} A_k(x,y)\right],
\]
if the last limit is finite. The modifications needed for $k<2$ are similar to those used above when dealing with the drift term.
 The remaining part is the same as in Theorem \ref{thstochmonkthorderdifint}.
\end{proof}

Let us now turn to processes with jumps starting with the generator
\begin{equation}
\label{eqjumpgen}
Lg(x)=\int (g(y)-g(x)) \nu (x, dy)
\end{equation}
with a finite stochastic kernel $\nu (x,dy)$.
Since
\[
L'\phi (dy)=\int_{z\in \R} \phi (dy) \nu (y, dz)-\phi (dy) \int_{z\in \R} \nu (y, dz)
\]
we find, for the dual of order $k>0$, the expression
\[
L^{D_k} g (y) =I_k^+ \circ  L' \circ \frac{d^k}{d(-y)^k} g(y)
\]
\[
=\int_y^{\infty} \frac{(z-y)^{k-1}}{\Ga (k)} \left[\int dw \frac{d^k}{d(-w)^k} g(w) \nu (w, dz) -\frac{d^k}{d(-z)^k} g(z) dz \int \nu (z, dw)\right],
 \]
 \[
 =\int \int \frac{d^k}{d(-z)^k} g(z) \nu (z, dw)
  \left[ \1_{w\ge y} \frac{(w-y)^{k-1}}{\Ga (k)} -\1_{z\ge y} \frac{(z-y)^{k-1}}{\Ga (k)} \right] dz.
  \]
  We shall now use the same trick, as when analyzing diffusion, by separating the part of the expression in the square brackets that would contributed to
  boundary terms after the $k$th order differentiation (everything is of course simpler for integer $k$). Thus we write
\[
L^{D_k} g (y) = \int \int \frac{d^k}{d(-z)^k} g(z)
\]
\[
\times  \left[\nu (z, dw) \1_{w\ge y} \frac{(w-y)^{k-1}}{\Ga (k)}
 +( \nu (y, dw)- \nu (z, dw))\1_{z\ge y} \frac{(z-y)^{k-1}}{\Ga (k)}
 -\nu (y, dw) \1_{z\ge y} \frac{(z-y)^{k-1}}{\Ga (k)} \right] dz.
\]
Integration by parts using \eqref{eqintpartsfrac1}, assuming that the kernel $\nu (x, dy)$ is $k$ times differentiable with respect to $x$ as a measure, yields

 \[
 L^{D_k}g (y) =- g(y) \int \nu (y, dw)
 \]
 \begin{equation}
 \label{eqdualkjump}
 +\int g(z)  \frac{\pa ^k }{\pa z^k} \int \left[ \nu (z, dw) \1_{w\ge y} \frac{(w-y)^{k-1}}{\Ga (k)}
 +( \nu (y, dw)- \nu (z, dw))\1_{z\ge y} \frac{(z-y)^{k-1}}{\Ga (k)} \right] dz.
 \end{equation}

 For integer $k$, the term containing $\nu(y,dw)$ in the square brackets becomes superfluous.
 In particular, for $k=1$ and $k=2$ this simplifies to
 \begin{equation}
 \label{eqdual1jump}
  L^{D_1}g (y) =\int g(z) \left[\1_{z<y} \int_{w\ge y} \frac{\pa \nu}{\pa z}(z,dw)
 -\1_{z\ge y} \int_{w< y} \frac{\pa \nu}{\pa z}(z,dw)\right] dz -g(y) \int \nu (y,dw),
 \end{equation}
 and
  \[
  L^{D_2}g (y) =\int g(z) dz \bigl[\1_{z<y} \int_{w\ge y} (w-y) \frac{\pa^2 \nu}{\pa z^2}(z,dw)
 \]
 \begin{equation}
 \label{eqdual2jump}
 +\1_{z\ge y} \left(\int_{w< y} (y-w) \frac{\pa ^2 \nu}{\pa z^2}(z,dw) + \frac{\pa ^2 }{\pa z^2}\int (z-w) \nu (z,dw)\right)\bigr]
 -g(y) \int \nu (y,dw)
 \end{equation}
 respectively.

For operator \eqref{eqdualkjump} to be conservative and conditionally positive, the function
\begin{equation}
 \label{eqdualkjumpboundcons0}
L(z,y)=\frac{\pa ^{k-1} }{\pa z^{k-1}} \int \left[ \nu (z, dw) \1_{w\ge y} \frac{(w-y)^{k-1}}{\Ga (k)}
 +( \nu (y, dw)- \nu (z, dw))\1_{z\ge y} \frac{(z-y)^{k-1}}{\Ga (k)} \right]
\end{equation}
has to be positive non-decreasing and has to satisfy the boundary conditions
\[
L(z,y)|_{z=-\infty}^{\infty}=\int \nu (y,dw).
\]
Simplest natural conditions ensuring the latter can be taken as follows:
\begin{equation}
 \label{eqdualkjumpboundcons1}
\lim_{z\to -\infty} \frac{\pa ^{k-1} }{\pa z^{k-1}} \int_y^{\infty}(w-y)^{k-1} \nu (z, dw)=0,
\end{equation}
\begin{equation}
 \label{eqdualkjumpboundcons2}
\lim_{z\to +\infty} \frac{\pa ^{k-1} }{\pa z^{k-1}} \int \left( \1_{w\ge y} (w-y)^{k-1}-(z-y)^{k-1}\right) \nu (z, dw)=0.
 \end{equation}

 Summarizing, we get the following.

\begin{theorem}
\label{thstochmonkthorderjump}
Suppose a Feller process $X_t^x$ is generated by operator \eqref{eqjumpgen} with a bounded positive kernel $\nu (x,dy)$ such that
its derivatives with respect to $x$ up to and including order $k$ exists as (possibly signed) stochastic kernels
and \eqref{eqdualkjumpboundcons1} holds. Then Markov dual of order $k$ exists if and only if function \eqref{eqdualkjumpboundcons0}
is nondecreasing and condition \eqref{eqdualkjumpboundcons2} holds. If this is the case the generator of the dual process is given by
 \[
  L^{D_k}g (y)=\frac{1}{\Ga (k)}\int (g(z)-g(y))
  \]
  \begin{equation}
 \label{eqdualkjump1}
  \times \frac{\pa ^k }{\pa z^k} \int \left[ \nu (z, dw) \1_{w\ge y} (w-y)^{k-1}
 +(\nu (y, dw)- \nu (z, dw))\1_{z\ge y} (z-y)^{k-1} \right] dz.
 \end{equation}
 If one has inequality $\le$ in \eqref{eqdualkjumpboundcons2}, rather than equality, then the dual exists as a sub-Markov process,
 the generator being given by
 \eqref{eqdualkjump}.
\end{theorem}

\begin{remark}
As we assumed maximum regularity, the measure of jumps of the dual process turns out to be absolutely continuous with respect to Lebesgue measure.
More generally, the dual generator would look like
\[
  L^{D_k}g (y)=\frac{1}{\Ga (k)}\int (g(z)-g(y))
  \]
  \begin{equation}
 \label{eqdualkjump1}
  \times d_z \frac{\pa ^{k-1} }{\pa z^{k-1}} \int \left[ \nu (z, dw) \1_{w\ge y} (w-y)^{k-1}
 +(\nu (y, dw)- \nu (z, dw))\1_{z\ge y} (z-y)^{k-1} \right].
 \end{equation}
  \end{remark}

We have built dual generators separately for diffusive and jump parts of the original generators.
For an arbitrary Feller process with a pseudo-differential generator the dual is constructed by putting these parts together.
As an example let us consider Markov processes that are martingales, that is, they have generators of the form

\begin{equation}
\label{eqgenmarting}
Lg(x)=a(x)\frac{d^2}{dx^2} +\int [g(y)-g(x)-g'(x)(y-x)] \nu (x, dy)
\end{equation}
with
\[
\int \min (|y-x|, (y-x)^2) \nu (x, dy) <\infty.
\]
As the duality of all orders reverses the sign of the drift, for self-duality one has necessarily the condition
\begin{equation}
\label{eqgenmartingselfducond}
\int (y-x) \nu (x, dy)=0,
\end{equation}
where the integral can be understood in the sense of the main value. Taking these into account and looking at formulas
\eqref{eqdual1jump} and \eqref{eqdual1jump} we arrive at the following.

\begin{theorem}
\label{thselfdual12mart}
Let $k=1$ or $k=2$ and let a Feller process $X_t^x$ be generated by an operator of type \eqref{eqgenmarting} with a continuously differentiable nonnegative function
$a(x)$ and continuously differentiable in $x$ L\'evy kernel $\nu$.
Then the process $X_t^x$ is self-dual of order $1$ or $2$, if \eqref{eqgenmartingselfducond} holds and
\begin{equation}
\label{eqgenmartingselfducond1}
\nu (y,dz)=\1_{z<y} d_z \int_{w\ge y} \nu (z,dw)
 -\1_{z\ge y} d_z \int_{w< y} \nu (z,dw),
\end{equation}
or
\begin{equation}
\label{eqgenmartingselfducond2}
\nu (y,dz)=\1_{z<y} d_z \int_{w\ge y} (w-y) \frac{\pa \nu}{\pa z} (z,dw)
 +\1_{z\ge y} d_z \int_{w< y}(y-w) \frac{\pa \nu}{\pa z} (z,dw)
\end{equation}
respectively.
\end{theorem}

Let us note finally that the relation between dual generators become more transparent in differential form (even though in this way some information of boundary behavior is lost).
For instance, differentiating equations \eqref{eqgenmartingselfducond1} and \eqref{eqgenmartingselfducond2} with respect to $y$ once or twice respectively, yields
the differential relations
\begin{equation}
\label{eqgenmartingselfducond1d}
d_y\nu (y,dz)=d_z \nu (z,dy),
\end{equation}
and
\begin{equation}
\label{eqgenmartingselfducond2d}
d_y\frac{\pa \nu}{\pa y} (y,dz)=d_z \frac{\pa \nu}{\pa z}\nu (z,dy),
\end{equation}
respectively. These equations take specially simple form
\begin{equation}
\label{eqgenmartingselfducondden}
\frac{\pa \nu}{\pa y} \nu (y,z)=\frac{\pa \nu}{\pa z} \nu (z,y), \quad
\frac{\pa^2 \nu}{\pa y^2} \nu (y,z)=\frac{\pa^2 \nu}{\pa z^2} \nu (z,y)
\end{equation}
respectively, if $\nu(x,dy)$ has a density, $\nu (x,y)$, with respect to Lebesgue measure.

\section{Time-nonhomogeneous extension}

Equation \eqref{eqpropdualopesemig} suggests the necessity to include time reversion when studying time-non\-homogeneous situation.
We mean just a simple time reversion around a deterministic time, not a more sophisticated general time reversion, as developed, say, in \cite{Dynk} or \cite{ChuWal}.

Let us recall that a family $U_{s,t}$, $0\le s \le t$, of transformations in $B(X)$ (or $C(X)$ or $C_{\infty}(X)$) for locally compact spaces $X,Y$
is called a (backward) propagator if $U_{t,t}$ is the identity operator and the chain rule, or the propagator equation, holds for $t\leq s\leq r$: $U^{t,s}U^{s,r}=U^{t,r}$.

Suppose ${U^{t,r}}$ is a strongly continuous backward propagator of bounded linear operators in $C_{\infty}(X)$ with a common dense invariant domain $D$.
 Let ${A_t}$, $t\geq0$, be a family of linear operators $D\mapsto B$ that are strongly continuous in $t$.
  Let us say that the family ${A_t}$ {\it generates} ${U^{t,r}}$ on $D$ if,
for any $f\in D$, the equations
\begin{equation}
\label{Generates}
\frac{d}{ds}U^{t,s}f = U^{t,s}A_sf, \quad \frac{d}{ds}U^{s,r}f = -A_sU^{s,r}f, \quad 0\leq t\leq s\leq r,
\end{equation}
hold for all $s$ with the derivatives taken in the topology of $B$, where for $s=t$ (resp. $s=r$) it is assumed to be only a right (resp. left) derivative.
The second equation (which in fact follows from the first one under mild natural conditions) implies by duality that, for any $T$,
\begin{equation}
\label{Generatesdual}
\frac{d}{ds}U'_{T-t,T-s} = -A'_{T-s}U'_{T-t,T-s}, \quad 0\leq s\leq t \le T,
\end{equation}
in the weak sense.

The time-nonhomogeneous counterpart of Proposition \ref{propdualsemi} (with the literally the same proof) reads as follows.

\begin{prop}
\label{propdualpropa}
(i) Let $f$ be a bounded measurable function separating measures on $X$ and $U_{s,t}$, $s\le t$, a (backward)
propagator of integral operators in $B(X)$ specified by the family of bounded signed kernel $p_{s,t}(x,dz)$ from $X$ to $X$.
Then, for any $T>0$,  the operators $U_{s,t}^{D(f,T)}$ in $F(\MC(X))$, $f$-dual to $U_{T-t,T-s}$, also form a propagator and
\begin{equation}
\label{eqpropdualopesemig1}
U_{s,t}^{D(f,T)}=F \circ U'_{T-t,T-s} \circ F^{-1}.
\end{equation}

(ii)
If the propagator $\{U_{s,t}\}$ is strongly continuous in $C_{\infty}(X)$ with an invariant domain $D$
and is generated by a family of operator $A_t :D \to C_{\infty}$, then
\[
\frac{d}{ds}U_{s,t}^{D(f,T)}g= - F \circ A'_{T-s} \circ F^{-1}U_{s,t}^{D(f,T)}g,
\]
that is, the generator of the propagator $\{U_{s,t}^{D(f,T)}\}$ is
\begin{equation}
\label{eqdualgengennonhom}
A_s^{D(f,T)}=F \circ A'_{T-s} \circ F^{-1}.
\end{equation}
\end{prop}

The propagator $\{U_{s,t}^{D(f,T)}\}$ will be called $(f,T)$-dual to $\{U_{s,t}\}$.

Let us describe the probabilistic analog of this duality.
We shall write $X_t^{x,s}$ for a Markov process at time $t>s$ with initial position $x$ at time $s$.
For a function $f$ on $X\times Y$ with two metric (or measurable) spaces $X,Y$ and a number $T$, we say that the Markov processes $Y_t^{y,s}$ in $Y$
 is $(f,T)$-dual
to the Markov process $X_t^{x,s}$ in $X$, if
 \begin{equation}
 \label{eqdeffTdual}
\E f(x,Y_t^{y,s}) =\E f(X_{T-s}^{x, T-t},y)
\end{equation}
for all $x\in X, y\in Y$ and $s\le t \le T$, where $\E$ on the left hand side and the right hand side correspond to the distributions
of processes $Y_t$ and $X_t$ respectively.
In particular case of $X=Y$ and $f(x,y)=\1_{x\ge y}$ (where $\ge$ is any measurable partial order on $X$) this reduces to
\begin{equation}
 \label{eqdefTdualorder}
\P (Y_t^{y,s}\le x) =\P (X_{T-s}^{x, T-t} \ge y).
\end{equation}

 Thus the duality of Markov processes is equivalent to the duality of their propagators. Proposition \ref{propdualpropa} implies that
 dual distributions to a Markov process automatically form a Markov family, as their transition operators form a propagator and hence
 satisfy the chain rule (or Chapman-Kolmogorov equation).

It is now clear that all our results have a natural  counterpart for time-dependent generators.
Namely, let us say that
a Markov process $Y_t^y$ is dual to $X_t^x$ of order $k$, $k\in \R$, if
\begin{equation}
\label{eqdefputcallsymmgennonhom}
\E  (x-Y_t^{y,s})_+^{k-1}=\E (X_{T-s}^{x, T-t} - y)_+^{k-1}.
\end{equation}
The characterization in terms of generators or stochastic monotonicity remains the same, once the time dependence is adjusted appropriately,
that is via \eqref{eqpropdualopesemig1} and \eqref{eqdualgengennonhom}.

\section*{Appendix}

For completeness, we deduce here the fundamental solutions of the generators of L\'evy stable motions and fractional derivative operators,
as well as the related integration by parts formulas.

Recall that the characteristic function of a $\be$-stable L\'evy motion for $\be \in (0,1)\cup (1,2)$ equals
\[
\exp \{ -t\si |p|^{\be} e^{i\pi \ga \, {\rm sgn}\, p/2}\}
\]
where $\si>0$ is the scale and $\ga $ is the skewness parameter satisfying the conditions $|\ga| \le \be$ or
$|\ga|\le 2-\be$ for $\be \in (0,1)$ or $\be \in (1,2)$ respectively.
For simplicity, we omit the discussion of a more complicated general case $\be =1$, and for this case will deal only with the symmetric case $\ga=0$,
 for which the above formulas remain valid.

Thus the generator $L_{\be, \ga, \si}$ of this L\'evy motion is the pseudo-differential
 operator with the symbol (denoted with some abuse of notation by the same letter)
\[
L_{\be, \ga, \si}(p)=-\si |p|^{\be} e^{i\pi \ga \, {\rm sgn}\, p/2}
\]
meaning that $L_{\be, \ga, \si}$ acts on the Fourier transform $\FC(f)(p)=\hat f(p)=\int e^{-ixp} f(x) \, dx$ of a function $f$
as the multiplication by $L_{\be, \ga, \si}(p)$.

The fundamental solution of the operator $L_{\be, \ga, \si}$ equals
\[
f(x)=\left[{\FC}^{-1} \frac{1}{L_{\be, \ga, \si}(p)}\right](x)
=-\frac{1}{2\pi} \int_{\infty}^{\infty}\frac{e^{ipx} dp}{\si |p|^{\be} e^{i\pi \ga \, {\rm sgn} \, p/2}}.
\]
In other words,
\[
f=-\frac{1}{\si} \FC^{-1}\left[  e^{-i\pi \ga/2} p_+^{-\be} + e^{i\pi \ga/2} p_-^{-\be}\right].
\]
Using known formulas for the Fourier transforms (in the sense of distributions) of one-sided powers (see e.g. \cite{GeSh}, p. 176),
that is
\[
(\FC^{-1} p_{\pm}^{\la})(x)=\frac{\pm i}{2 \pi} e^{\pm i \la \pi /2} \Ga (\la +1) (x\pm i0)^{-\la-1}
\]
 (slight deviations in our notation from the Fourier transform used in \cite{GeSh} are taken into account),
where
\[
 (x\pm i0)^{\la}=x_+^{\la} +e^{\pm i \la \pi} x_-^{\la},
 \]
we find that
\[
f(x)=-\frac{\Ga (1-\be)}{2\pi \si} \left[i e^{-i (\be+\ga) \pi /2} (x+i0)^{\be -1}-i e^{i (\be+\ga) \pi /2} (x-i0)^{\be -1}\right]
\]
leading to
\begin{equation}
 \label{eqfundsolfraclevy}
f(x) =-\frac{\Ga (1-\be)}{\si \pi} \left[\sin (\pi (\be +\ga)/2) x_+^{\be -1} +\sin (\pi (\be -\ga)/2) x_-^{\be -1} \right].
\end{equation}

The most important cases are the fully skewed motions with $\ga =\pm \be$ or $\ga =\pm (2-\be)$ for $\be \in (0,1)$ or $\be \in (1,2)$ respectively,
and the symmetric motions with $\ga =0$. In these cases for $\be \in (0,1)$ the generators are negations of fractional
derivatives, that is

\begin{equation}
 \label{eqdeffractionalderivative2}
L_{\be, \pm \be, 1}f=- \frac{d^{\be}}{d( \pm x)^{\be}}f(x) = - \frac{1}{\Gamma (-\be)}\int_0^{\infty} (f(x \mp y)-f(x))\frac {dy}{y^{1+\be}},
\end{equation}
and
\begin{equation}
 \label{eqdeffractionalderivative3}
L_{\be, 0, 1}f =- \left|\frac{d}{dx}\right|^{\be} = -\frac{1}{2 \cos (\pi \be /2)} \left(\frac{d^{\be}}{d x^{\be}}+\frac{d^{\be}}{d(-x)^{\be}}\right),
\end{equation}
see e.g. Sect 1.8 in \cite{Kobook11}, where the fractional derivatives can be defined either by the corresponding expressions on the r.h.s.
of these formulas, or
 equivalently via the following Fourier transforms:

\begin{equation}
 \label{eqFourierforfracionalder1}
 \FC(\frac{d^{\be}}{d(\pm x)^{\be}}f)(p)=\exp \{\pm i\pi \be \, {\rm sgn} \,
 p/2\}|p|^{\be} \FC(f)(p),
\end{equation}
\begin{equation}
 \label{eqFourierforfracionalder2}
 \FC(\frac{d^{\be}}{d x^{\be}}+\frac{d^{\be}}{d(-x)^{\be}})(p)
 = 2\cos (\pi \be /2) |p|^{\be}\FC(f)(p).
\end{equation}

The processes generated by $-\frac{d^{\be}}{d(\pm x)^{\be}}$ with $\be \in (0,1)$ are called the stable L\'evy subordinators.

%Formula \eqref{eqdeffractionalderivative3} remains valid for $\be =1$.

From \eqref{eqfundsolfraclevy} it follows that the functions
\begin{equation}
 \label{eqfundsolfrac1}
\frac{\Ga (1-\be)}{\pi} \sin (\pi \be) x_{\pm}^{\be -1}=\frac{x_{\pm}^{\be -1}}{\Ga (\be)}
\end{equation}
(where the equation $\Ga (\be)\Ga (1-\be)=\pi/ \sin (\pi \be)$ was used),
represent fundamental solutions for the operators $d^{\be}/d(\pm x)^{\be}$.

Fractional derivatives of order higher than $1$ can be defined by the compositions of the derivatives of order $\be \in (0,1)$ with the
derivatives of an integer order. Namely, for $\be \in (n,  n+1)$, $n\in \N$, one defines
\[
\frac{d^{\be}}{d( \pm x)^{\be}}=\frac{d^n}{d(\pm x)^n}\frac{d^{\be-n}}{d( \pm x)^{\be-n}}
\]
with the second component given by \eqref{eqdeffractionalderivative2}. It is then easy to check that formulas \eqref{eqfundsolfrac1} for the fundamental solutions remain valid for all $\be=k+1>0$, $\be \notin \N$.
Formula \eqref{eqdeffractionalderivative3} remains valid for $\be \in (1,2)$, but in \eqref{eqdeffractionalderivative2} the sign has to be changed
leading to
\begin{equation}
 \label{eqdeffractionalderivative4}
L_{\be, \pm (2-\be), 1}f=\frac{d^{\be}}{d( \pm x)^{\be}}f(x)
= \frac{1}{\Gamma (1-\be)}\int_0^{\infty} ((\pm f')(x \mp y)-(\pm f)' (x))\frac {dy}{y^{\be}}.
\end{equation}

Turning to the integration by parts, let us define, for $k\ge 1$,  the fractional integration operator $I_k^{\pm}: \MC_k^{\pm}(\R) \to B(\R)$, by the equation
\begin{equation}
 \label{eqdeffracint1}
(I_k^{\pm}Q)(y)=\int \frac{(x-y)^{k-1}_{\pm}}{\Ga (k)} Q(dx),
\end{equation}
where
\[
\MC_k^{\pm}=\{Q\in \MC(\R):  \int x_{\pm}^{k-1} |Q| (dx) <\infty \},
\]
and
\[
 x_+^k= \left\{
\begin{aligned}
& x^k, \quad x\ge 0 \\
& 0, \quad x<0.
\end{aligned}
\right. , \quad x_-^k(x)= (-x)_+^k.
\]
The image $\IC^{\pm}_1$ of $I_1^{\pm}$ (defined on $\MC_1^{\pm}=\MC(\R)$) is the set of right- (respectively left-) continuous
functions of finite total variation, tending to zero at $\pm \infty$. Moreover, $(I_1^{\pm}Q)'=\mp Q$ in the sense of distributions.
The image $\IC_k^{\pm}$ of $I_k^{\pm}$, $k>1$, consists of continuous functions $g$ tending to zero at $\pm \infty$ and such that
\[
\frac{d^k}{d(\mp x)^k} g \in \MC_k^{\pm} (\R)
\]
in the sense of distributions. Moreover,
\[
 \frac{d^k}{d(\mp x)^k} \circ I_k^{\pm}, \quad I_k^{\pm} \circ  \frac{d^k}{d(\mp x)^k}
 \]
 are the identity operators in $\MC_k^{\pm} (\R)$ and  $\IC_k^{\pm}$ respectively. Other simple formulas worth mentioning are
\begin{equation}
 \label{eqintpartsfrac0}
I_k^{\pm} \circ \frac{d}{d(\mp x)} = \left\{
\begin{aligned} & I_{k-1}^{\pm}, \quad k>1, \\
& d^{1-k}/d(\mp x)^{1-k}, \quad k<1,
\end{aligned}
\right.
\end{equation}
\begin{equation}
 \label{eqintpartsfrac00}
\frac{d^k}{dx^k}\frac{(x-a)^{k-1}_+}{\Gamma (k)}=\de_a(x), \quad k>1.
\end{equation}

 Let now $\phi_{\pm} =I_k^{\pm} Q_{\pm}$ with some $Q_{\pm} \in \MC^{\pm}_k$. By Fubini's theorem
 \[
 \int_{\R^2}  \frac{(x-y)^{k-1}_{+}}{\Ga (k)} Q_+(dx)Q_-(dy)
  =\int (I_k^+Q_+)(y)Q_-(dy)=\int (I_k^-Q_-)(x)Q_+(dx).
  \]
The last equation can be called the integration-by-parts formula, as it rewrites as
\begin{equation}
 \label{eqintpartsfrac1}
\int \phi_+(y)\frac{d^k}{dy^k} \phi_- (dy)= \int \phi_-(x)\frac{d^k}{d(-x)^k} \phi_+ (dx)
\end{equation}
(where the derivatives are defined, generally speaking, in the sense of distributions and represent measures, not necessarily functions).

It is important to stress that this formula holds not only for the integration over $\R$, but also for the integration over an
interval or a half-line, the corresponding boundary terms being taken into account automatically
by the measures $d^k \phi_{\pm} (x)/d(\mp x)^k$.

Finally, for $k\in (0,1)$ we can define fractional integration \eqref{eqdeffracint1} on functions $g\in B(\R)\cap L^1(\R)$, that is as
\begin{equation}
 \label{eqdeffracint2}
(I_k^{\pm}g)(y)=\int \frac{(x-y)^{k-1}_{\pm}}{\Ga (k)} g(x) \, dx,
\end{equation}
in which case the image belongs to the set of continuous functions with the sup-norm bounded by $\|g\|/k+\|g\|_{L^1}$ and
\eqref{eqintpartsfrac1} still holds by the same reasoning.

\end{document}